\documentclass[11pt]{amsart}
\usepackage{amsmath,amssymb,latexsym,soul,cite,mathrsfs,amsfonts}
\usepackage{float}
\usepackage{xcolor}
\usepackage{color,enumitem,graphicx}
\usepackage[colorlinks=true,urlcolor=blue,
citecolor=blue,linkcolor=blue,linktocpage,pdfpagelabels,
bookmarksnumbered,bookmarksopen]{hyperref}
\usepackage[english]{babel}
\usepackage{outline}
\usepackage{comment}
\usepackage{frame}
\usepackage[left=1.9cm,right=1.9cm,top=2.0cm,bottom=2.0cm]{geometry}
\usepackage[hyperpageref]{backref}
\usepackage[colorinlistoftodos]{todonotes}
\makeatletter
\providecommand\@dotsep{5}
\def\listtodoname{List of Todos}
\def\listoftodos{\@starttoc{tdo}\listtodoname}
\makeatother
\usepackage{verbatim}
\numberwithin{equation}{section}

\newtheorem{thm}{Theorem}[section]
\newtheorem{prop}[thm]{Proposition}
\newtheorem{lem}[thm]{Lemma}

\newtheorem{rem}{Remark}
\newtheorem{definition}{Definition}

\newcommand{\R}{\mathbb{R}}
\newcommand{\1}{\frac{1}{2}}
\newcommand{\2}{2^*_{\gamma,\mu}}
\newcommand{\diam}{\operatorname{diam}}

	\title[Bifurcation and multiplicity results for critical \dots]{Bifurcation and multiplicity results for critical Grushin-Choquard problems}
	\author[Suman Kanungo]{Suman Kanungo}
	
	\address[]{\newline\indent
		Department of Mathematics
		\newline\indent
		Indian Institute of Technology Bhilai
		\newline\indent
		491002, Durg, Chhattisgarh, India}
	\email{\href{sumankau@iitbhilai.ac.in}{sumankau@iitbhilai.ac.in}}
	
	\author[Pawan Kumar Mishra]{Pawan Kumar Mishra}
	\address[]{\newline\indent
		Department of Mathematics
		\newline\indent
		Indian Institute of Technology Bhilai
		\newline\indent
		491002, Durg, Chhattisgarh, India}
	\email{\href{pawan@iitbhilai.ac.in}{pawan@iitbhilai.ac.in}}

	\author[Giovanni Molica Bisci]{Giovanni Molica Bisci}
	\address[]{\newline\indent  
		Dipartimento di Promozione delle Scienze Umane
		\newline\indent e
		della Qualità della Vita
		\newline\indent
		Università Telematica San Raffaele Roma
		\newline\indent
		00166, Via di Val Cannuta 247, Roma, Italy}
	\email{\href{giovanni.molicabisci@uniroma5.it}{giovanni.molicabisci@uniroma5.it}}
	
	
\begin{document}
	\begin{abstract}
		We consider the following nonlocal  Br\'ezis-Nirenberg type critical Choquard problem involving the Grushin operator
		\begin{equation*}
			\left\{
			\begin{aligned}
				-\Delta_\gamma & u =\lambda u + \left(\displaystyle\int_\Omega \frac{|u(w)|^{\2}}{d(z-w)^\mu}dw\right) |u|^{\2-2}u \quad &&\text{in} \ \Omega,\\
				u &= 0 \quad &&\text{on} \, \partial \Omega,
			\end{aligned}
			\right.
		\end{equation*}
		where $\Omega$ is an open bounded domain in $\mathbb{R}^N$, $N \geq 3$, with $\Omega \cap \{ x=0\} \neq \emptyset$, and $\lambda >0$ is a parameter. Here, $\Delta_\gamma$ represents the Grushin operator, defined as
		\[
		\Delta_\gamma u(z) = \Delta_x u(z) +(1+\gamma)^2 |x|^{2\gamma} \Delta_y u(z), \quad \gamma \geq 0,
		\]
		where $z=(x,y)\in \Omega \subset \R^m\times \R^n$, $m+n=N \geq 3$ and $\2= \frac{2N_\gamma-\mu}{N_\gamma-2}$ is the Sobolev critical exponent in the Hardy-Littlewood context with $N_\gamma= m+(1+\gamma)n$ is the homogeneous dimension associated to the Grushin operator and $0<\mu<N_\gamma$. The homogeneous norm related to the Grushin operator is denoted by
		$d(\cdot)$. In this article, we prove the existence of bifurcation from any eigenvalue $\lambda^*$ of $-\Delta_\gamma$ under Dirichlet boundary conditions. Furthermore, we show that in a suitable left neighborhood of $\lambda^*$, the number of nontrivial solutions to the problem is at least twice the multiplicity of $\lambda^*$.
		
	\end{abstract} 
	\maketitle
	\section{Introduction} \label{sec:1}

	Let $z \in \R^N$, we can write $z=(x,y)$, where $x\in \R^m$ and $y \in \R^n$ with $m$ and $n$ being two non-negative integers such that $N=m+n \geq 3$. Now, we consider the Grushin operator $\Delta_\gamma$ ($\gamma\geq 0$), introduced by Baouendi \cite{Baouendi} and later studied by Grushin \cite{Grushin}, defined by
	\begin{align*}
		\Delta_\gamma u(z) = \Delta_x u(z) + (1+\gamma)^2|x|^{2\gamma} \Delta_y u(z),
	\end{align*}
	where $\Delta_x$ and $\Delta_y$ are, respectively, the Laplace operators with respect to the variables $x$ and $y$, we observe that it reduces to the Laplace operator for $\gamma = 0$. The Grushin operator is a generalization of the degenerate differential operator \cite{Tricomi} $\frac{\partial^2}{\partial x^2} + x \frac{\partial^2}{\partial y^2}$. $\Delta_\gamma$ belongs to the class of subelliptic operators, intermediate between elliptic and hyperbolic operators (see \cite{Egorov}).
	One of the key properties of this operator is its lack of uniform ellipticity in $\mathbb{R}^N$, as it is degenerate on the subspace $\{0\} \times \mathbb{R}^n$. The homogeneous norm $d(\cdot)$ related to the operator is defined by
	\[  d(z) = ( |x|^{2(\gamma+1)}+ |y|^2)^{\frac{1}{2(\gamma+1)}} \]
	and the anisotropic dilations naturally associated with the Grushin operator are defined as
	\[
	\tau_t(x, y) = \left( t x, \, t^{\gamma+1} y \right) \quad \text{for all} \ (x, y) \in \mathbb{R}^m \times \mathbb{R}^n, \ t > 0.
	\]
	It is well known that the operator $\Delta_\gamma$ is homogeneous of degree two with respect to these dilations, while the norm $d(z)$ is homogeneous and of degree one. Moreover, there exists a constant $C > 0$, depending only on $\gamma$ and the homogeneous dimension $N_\gamma= m+(1+\gamma)n$, such that
	\[
	\Phi(z) = \frac{C}{d(z)^{N_\gamma - 2}}
	\]
	is the fundamental solution of $-\Delta_\gamma$ with a singularity at the origin (see \cite[Proposition 2.1]{Garofalo}). Throughout this paper we use the following notation, $B_r = B_r(0)$ is a ball of radius $r>0$, centered at the origin, respect to the homogeneous norm $d(z)$, i.e.,
	\[
	B_r = \left\{ z \in \mathbb{R}^N \, : \, d(z) < r \right\}.
	\]

	Motivated by the seminal work of Brezis and Nirenberg \cite{Brezis}, the following critical problem has been extensively studied:
	\begin{equation}
		\label{1.1}
		\left\{
		\begin{aligned}
			-\Delta &u = \lambda u + |u|^{2^* - 2} u \quad &&\text{in} \ \Omega,\\  u& = 0  &&\text{ on }  \partial \Omega,
		\end{aligned}
		\right.
	\end{equation}
	where $2^* = \frac{2N}{N-2}$ is the critical Sobolev exponent, $\Omega$ is an open bounded domain in $\R^N$ ($N \geq 3$), and $\lambda > 0$ is a positive parameter. The bifurcation and multiplicity of problem \eqref{1.1} were first studied by Cerami, Fortunato, and Struwe \cite{Cerami}. In their work, the authors proved that there is a bifurcation from any eigenvalue of $-\Delta_\gamma$ with Dirichlet boundary conditions. Moreover, they showed that in a suitable left neighborhood of any such eigenvalue, the number of solutions is at least twice the multiplicity of the corresponding eigenvalue. They also provided an explicit estimate for the size of this neighborhood, which depends on the best constant in the Sobolev inequality, the Lebesgue measure of the domain, and the dimension. Following this seminal contribution, similar results have been obtained for a variety of variational operators. For instance, Fiscella, Molica Bisci, and Servadei \cite{Fiscella} extended these ideas to the nonlocal framework by considering a general class of fractional operators. Their work generalizes the bifurcation and multiplicity results of Cerami, Fortunato, and Struwe \cite{Cerami} from classical elliptic equations to the setting of nonlocal fractional operators. We refer to \cite{Perera} for quasilinear results, where the authors extended the bifurcation and multiplicity theory for the critical fractional $p$-Laplacian from the semilinear case
	$p=2$ to all $p \in (1, \infty)$ by using a new abstract critical point theorem.
	
	On the other hand, Choquard-type problems have gained significant attention in recent years. The equation
	\begin{equation}
		\label{C}
		-\Delta u + V(x)u = \left( \int_{\R^N} \frac{|u(w)|^p}{|z-w|^\mu} \ dw \right) |u|^{p-2} u \quad \text{in} \ \R^N,
	\end{equation}
	with $N=3$, $\mu=1$, $p=2$ and $V$ as a continuous potential, was first introduced by
	Pekar \cite{Pekar} in 1954 in the context of quantum field theory to describe a polaron at rest. Later, P. Choquard \cite{Lieb} also studied \eqref{C} as an approximation to the Hartree-Fock theory for one component plasma, modeling an electron trapped in its hole. Since then, several authors have investigated Choquard-type problems with $N \geq 3, 0 < \mu < N, p \geq 2$ as seen in the works of \cite{Ackermann,Lions,Moroz,Moroz2,Moroz3} and the references therein. In particular, Gao and Yang \cite{Gao} established the existence of solutions for the Br\'ezis-Nirenberg problem \eqref{1.1} with Choquard-type nonlinearity
	\begin{equation*}
		\left\{
		\begin{aligned}
			-\Delta & u =\lambda u + \left(\displaystyle\int_\Omega \frac{|u(w)|^{2^*_\mu}}{|z-w|^\mu}dw\right) |u|^{2^*_\mu-2}u \quad &&\text{in} \ \Omega,\\
			u &= 0 \quad &&\text{on} \, \partial \Omega,
		\end{aligned}
		\right.
	\end{equation*}
	where $2^*_\mu$ is the critical exponent in the sense of the Hardy-Littlewood-Sobolev inequality. For the fractional Laplacian counterpart of this problem, similar results were obtained by Mukherjee and Sreenadh \cite{Mukherjee}.
	
	Now, coming to the Grushin operator $\Delta_\gamma$, we observe that it is contained in a family of operators of the form
	\[
	\Delta_\lambda := \sum_{i=1}^{N} \partial_{x_i} \left( \lambda_i^2 \partial_{x_i} \right), \quad \lambda = (\lambda_1, \dots, \lambda_N),
	\]
	which was introduced and studied by Franchi and Lanconelli \cite{Franchi} and later discussed by Kogoj and Lanconelli \cite{Kogoj}. Problems involving the Grushin operator have been widely studied in the last few years. For interested readers, we refer \cite{Holanda,Alves,Loiudice,Monti1,Monti2} along with the references therein. Recently, Molica Bisci, Malanchini, and Secchi \cite{Bisci} established a bifurcation and multiplicity results for the Br\'ezis-Nirenberg problem \eqref{1.1} involving the Grushin operator
	\begin{equation}
		\nonumber
		\left\{
		\begin{aligned}
			-\Delta_\gamma &u = \lambda u + |u|^{2^*_\gamma - 2} u \quad &&\text{in} \ \Omega,\\  u& = 0  &&\text{ on }  \partial \Omega,
		\end{aligned}
		\right.
	\end{equation}
	where $2^*_\gamma = \frac{2N_\gamma}{N_\gamma-2}$ is the associated Sobolev critical exponent. Their results extend those obtained in \cite{Cerami} for classical elliptic problems and in \cite{Fiscella} for nonlocal fractional operators. For related results concerning the
	$p$-Grushin operator, see \cite{Bisci2}.
	
	Motivated by the above results for classical elliptic operators, fractional operators, and the Grushin framework, our aim is to establish bifurcation and multiplicity results for the Br\'ezis-Nirenberg type Choquard problem involving the Grushin operator
	\begin{equation}
		\label{a}
		\left\{
		\begin{aligned}
			-\Delta_\gamma & u =\lambda u + \left(\int_\Omega \frac{|u(w)|^{\2}}{d(z-w)^\mu}dw\right) |u|^{\2-2}u \quad &&\text{in} \ \Omega,\\
			u &= 0 \quad &&\text{on} \, \partial \Omega,
		\end{aligned}
		\right.
	\end{equation}
	where $ \Omega $ is an open, bounded domain in $ \mathbb{R}^N $, $ N \geq 3 $, with $\Omega \cap \{ x=0\} \neq \emptyset$, $\lambda >0$ is a parameter and $0<\mu<N_\gamma$. Here, $\2= \frac{2N_\gamma-\mu}{N_\gamma-2}$ is the Sobolev critical exponent in the Grushin Hardy-Littlewood context.
	
	Now, we define the function space and variational setup for our problem.
	
	\subsection{Function space and variational formulation} 
	Consider the space
	\[  \mathcal{H}_\gamma(\R^N) := \left\{ u\in L^2(\R^N): \nabla_x u, |x|^\gamma \nabla_y u \in L^2(\R^N)\right\}   \]  endowed with the norm $\left( \int_{\R^N} (|\nabla_\gamma u|^2 + |u|^2) \ dz\right)^{\1}$, 
	where the Grushin gradient $\nabla_\gamma $ is defined as
	\begin{align*}
		\nabla_\gamma u(z) := \left( \nabla_x u(z), (1+\gamma)|x|^\gamma \nabla_y u(z) \right).
	\end{align*} 
	For a bounded set $\Omega \subset \R^N$, we define the space $\mathcal{H}_\gamma(\Omega)$ as the completion of $C_c^\infty(\Omega)$ with respect to the norm $\|u\|_{\gamma}: = \left( \int_{\Omega} |\nabla_\gamma u|^2 dz\right)^{\1}$. 
	The space $\mathcal{H}_\gamma(\Omega)$ is a Hilbert space with the inner product $$\langle u,v\rangle_\gamma := \int_\Omega \nabla_\gamma u \nabla_\gamma v \ dz.$$

	From \cite{Holanda,Monti1}, the space $\mathcal{H}_\gamma(\mathbb{R}^N)$ is continuously embedded into Lebesgue spaces, namely,
	\[
	\mathcal{H}_\gamma(\mathbb{R}^N) \hookrightarrow L^p(\mathbb{R}^N),
	\quad p \in [2,2_\gamma^*].
	\]
	For a bounded domain $\Omega \subset \mathbb{R}^N$ with $\Omega \cap \{ x=0\} \neq \emptyset$, the space $\mathcal{H}_\gamma(\Omega)$ is continuously embedded into $L^p$ spaces, i.e.,
	\[
	\mathcal{H}_\gamma(\Omega) \hookrightarrow L^p(\Omega), 
	\quad p \in [1,2_\gamma^*],
	\]
	and the embedding is compact for $p \in [1,2_\gamma^*),$
	see \cite{Holanda,Kogoj, Loiudice1}.
	
	\begin{definition}
		We say that a function $u \in \mathcal{H}_\gamma(\Omega)$ to be a (weak) solution of \eqref{a} if
		\[ \int_\Omega \nabla_\gamma u \nabla_\gamma \phi \ dz = \lambda \int_\Omega u \phi \ dz +\int_\Omega \left( \int_\Omega \frac{|u(w)|^{\2}}{d(z-w)^{\mu}}dw \right) |u|^{\2-2}u\phi \ dz, \]
	\end{definition}
	for all $\phi \in \mathcal{H}_\gamma(\Omega)$. The energy functional $J_\gamma: \mathcal{H}_\gamma(\Omega) \rightarrow \R$ associated to problem \eqref{a} is given by
	\begin{align*}  J_\gamma(u) &= \frac{1}{2} \int_\Omega |\nabla_\gamma u|^2dz - \frac{\lambda}{2}\int_\Omega |u|^2dz \\
		& \quad-\frac{1}{2 \cdot \2} \int_\Omega \left( \int_\Omega \frac{|u(w)|^{\2}}{d(z-w)^{\mu}}dw \right) |u(z)|^{\2} dz .\end{align*}

	\subsection{Spectral properties of the Grushin operator}
	We denote $0<\lambda_1 \leq \lambda_2 \leq \lambda_3 \leq \cdots \leq \lambda_j \leq \cdots $ the sequence of eigenvalues of the problem
	\begin{align*}
		-\Delta_\gamma u = \lambda u \ \text{ in } \ \Omega, \ u=0 \ \text{ on } \ \partial \Omega
	\end{align*} with $\lambda_j \rightarrow \infty$ as $j \rightarrow \infty$.
	The eigenfunctions $\{ \phi_j\}_{j \geq 1}$ corresponding to each $\lambda_j$ forms an orthonormal basis for $L^2(\Omega)$ and an orthogonal basis for $\mathcal{H}_\gamma(\Omega)$ (see \cite{Xu}). The characterization of the first eigenvalue $\lambda_1$ is given by
	\begin{equation*}
		\lambda_1 = \min_{u\in \mathcal{H}_\gamma(\Omega)\setminus \{0\}} \frac{\int_\Omega|\nabla_\gamma u |^2 dz}{\int_\Omega |u|^2 dz }.
	\end{equation*}
	Furthermore, for any $j\in \mathbb{N}$, the eigenvalues can be characterized as follows:
	\begin{equation*}
		\lambda_{j+1} = \min_{u\in \mathcal{H}_j^{\perp}(\Omega)\setminus \{0\}} \frac{\int_\Omega|\nabla_\gamma u |^2 dz}{\int_\Omega |u|^2 dz },
	\end{equation*}
	where $\mathcal{H}_j^{\perp}(\Omega) = \{ u \in \mathcal{H}_\gamma(\Omega): \langle u, \phi_i \rangle_\gamma, \text{ for } i= 1, 2, \cdots, j \}$. We denote $\mathcal{H}_{\gamma,j}(\Omega)= span\{\phi_1, \phi_2, \cdots, \phi_j\},$ then $\mathcal{H}_{\gamma,j}(\Omega) \bigoplus \mathcal{H}_j^{\perp}(\Omega)=\mathcal{H}_\gamma(\Omega)$ and $\mathcal{H}_{\gamma,j}(\Omega)$ is finite dimensional.
	
	\subsection{Main result} To prove the multiplicity of solutions for \eqref{a}, we use the abstract critical point theorem due to P. Bartolo, V. Benci, and D. Fortunato \cite{Bartolo}.
	\begin{thm}
		\label{thm1}
		Let $I : H \rightarrow \R$ be a $C^1$ functional on a Hilbert space $(H, \|\cdot\|)$ satisfying the following conditions:
		\begin{enumerate}
			\item[(i)] $I(u)=I(-u)$ and $I(0)=0$;
			\item[(ii)] there exists a positive constant $\beta$ such that $I$ satisfies the Palais-Smale condition in $(0, \beta)$;
			\item[(iii)] there exist two closed subspaces $V$ and $W$ of $H$ and three positive constants $\rho, \delta, \beta'$ with $\delta<\beta'<\beta$ such that
			\begin{enumerate}
				\item[(a)] $I(u)\leq \beta'$ for all $u\in W$,
				\item[(b)] $I(u) \geq \delta$ for any $u\in V$ with $\|u\|=\rho$;
				\item[(c)] codim $V< \infty$ and dim $W \geq codim V$.
			\end{enumerate}
		\end{enumerate}
		Then, there exist at least $dim W-codim V$ pairs of critical points of $I$, with critical values belonging to the interval $[\delta, \beta']$.
	\end{thm}
	
	Our main result can be stated as follows:
	\begin{thm}
		\label{m}
		Let $\Omega$ be an open, bounded subset of $\R^N, N \geq 3$ with smooth boundary and $\Omega \cap \{ x=0\} \neq \emptyset$. Let $\lambda >0$ and $\lambda^*$ be the eigenvalue of the operator $(-\Delta_\gamma, \mathcal{H}_\gamma(\Omega))$ given by
		\[ \lambda^*= \min_{k\in \mathbb{N}}\{\lambda_k : \lambda < \lambda_k \} \]
		and let $m\in \mathbb{N}$ be its multiplicity. Assume that
		\[ \lambda^* - \lambda < \frac{S_\Omega}{|\Omega|^{\frac{\2-2}{\2}}(\operatorname{diam} \Omega)^{\frac{\mu}{\2}}},\]
		where $\diam \Omega$ is the diameter of the domain $\Omega$ with respect to the Grushin norm $d(\cdot)$ and $S_\Omega$ is defined in \eqref{1.5}. Then, problem \eqref{a} admits at least $m$ pairs of nontrivial solutions $\{ -u_{\gamma,i}, u_{\gamma,i}\}$, $i = 1, 2, \cdots, m$ such that $\|u_{\gamma, i}\|_\gamma \to 0$ as $\lambda \to \lambda^*$.
	\end{thm}
	
	The main contributions of this paper are as follows:\begin{itemize}
		\item [(i)] To the best of our knowledge, a Hardy-Littlewood-Sobolev inequality in the context of the Grushin operator has not been available in the literature. In this paper, we establish such an inequality, which plays a crucial role in addressing the critical Choquard-type nonlinearity.
		\item [(ii)] We also introduce a new critical level for the compactness of Palais-Smale sequences and establish bifurcation and multiplicity results using the abstract critical point theorem due to \cite{Bartolo}.
		\item [(iii)] Our work extends  \cite{Bisci,Cerami,Fiscella} to the nonlocal Grushin-Choquard problems.
		\item [(iv)] Results of this paper are new even for the case $\gamma=0$.
	\end{itemize}
	
	The outline of the paper is as follows. In Section \ref{sec1}, we establish a version of the Hardy-Littlewood-Sobolev inequality in the Grushin setting. In Section \ref{sec2}, we present preliminary results that are essential for proving the main theorem, including a compactness result for Palais-Smale sequences. In Section \ref{sec3}, we prove some geometric properties of the functional and provide proof of the main theorem.
	
	\section{Hardy-Littlewood-Sobolev inequality} \label{sec1}
	
	\setcounter{section}{2} \setcounter{equation}{0} 
	In order to address Choquard-type problems, the Hardy-Littlewood-Sobolev inequality plays a crucial role. As far as we know, a complete proof of this inequality in the context of the Grushin operator does not appear in the literature. Therefore, we generalize the classical results of \cite{Stein} to the setting of the Grushin operator. We include a detailed outline of its proof for completeness.
	
	Define the following integral operator:
	\[   I_{\gamma, \mu} f(z) = \int_{\R^N} \frac{f(w)}{d(z-w)^\mu} \ dw \ \text{ with } \mu \in (0, N_\gamma). \]
	
	In order to prove the Hardy-Littlewood-Sobolev (HLS) inequality, we first need to understand when the operator $I_{\gamma, \mu}$ is bounded between Lebesgue spaces.
	
	For this reason, we consider the following question. Given $\mu$ with $0 < \mu < N_\gamma$, for which pairs $p$ and $q$ is the operator $f \mapsto I_{\gamma,\mu}(f)$ bounded from $L^p(\mathbb{R}^n)$ to $L^q(\mathbb{R}^n)$?
	That is, when does the inequality
	\begin{equation}
		\label{eq2}
		\|I_{\gamma,\mu}(f)\|_{q} \leq C \|f\|_{p}
	\end{equation}
	hold for some constant $C > 0$?
	
	There is a basic necessary condition for the boundedness of the operator $I_{\gamma, \mu}$, which comes from the fact that its kernel $d(z)^{-\mu}$ is homogeneous. To understand this, we consider how both sides of the equation \eqref{eq2} behave under scaling. For each $t>0$, define the dilation $\tau_t: \R^{N} \to \R^{N}$ by
	\[
	\tau_t f(z) := f(tx, t^{\gamma+1}y).
	\]
	By change of variable, we get
	\begin{equation*}
		\|\tau_tf(z)\|_p= t^{-\frac{N_\gamma}{p}} \|f\|_p \quad \text{ and } \quad \|I_{\gamma,\mu} (\tau_t f)\|_q = t^{\mu-N_\gamma-\frac{N_\gamma}{q}} \|I_{\gamma,\mu}(f)\|_q.
	\end{equation*}
	Hence, for the inequality \eqref{eq2} to hold, a necessary relation between the exponents $p$ and $q$ is
	\[ \frac{1}{q}= \frac{1}{p} + \frac{\mu}{N_\gamma} -1. \]
	In the next result, we show that this condition is also sufficient for the inequality \eqref{eq2} to hold.
	\begin{thm}
		\label{lem1.2}
		Let $0 < \mu < N_\gamma$, and let $1 \leq p < q < \infty$ satisfy
		$\frac{1}{q} = \frac{1}{p} + \frac{\mu}{N_\gamma}-1.$
		Then the following hold:
		\begin{itemize}
			\item[(a)] If $f \in L^p(\mathbb{R}^N)$, then the integral $I_{\gamma, \mu}f(z)$ converges absolutely for almost every $z \in \mathbb{R}^N$.
			
			\item[(b)] If $1< p $, then the operator $I_{\gamma, \mu}$ is bounded from $L^p(\mathbb{R}^N)$ to $L^q(\mathbb{R}^N)$, i.e.,
			\[
			\| I_{\gamma, \mu}(f) \|_{q} \leq C \|f\|_{p},
			\]
			where $C$ is a constant independent of $f$.
			
			\item[(c)] If $f \in L^1(\mathbb{R}^N)$, and $q = \frac{N_\gamma}{\mu}$, then for all $\kappa > 0$,
			\[
			\left|\left\{ z \in \mathbb{R}^N : |I_{\gamma, \mu}(f)(z)| > \kappa \right\}\right|
			\leq \left(\frac{C\|f\|_{1}}{\kappa} \right)^q.
			\]
			That is, the operator $I_{\gamma, \mu}$ is of weak-type $(1, \frac{N_\gamma}{\mu})$.
		\end{itemize}
	\end{thm}
	\begin{proof}
	For any $\alpha > 0$, we decompose the operator $I_{\gamma,\mu}$ as $I^1_{\gamma,\mu}+I^2_{\gamma,\mu}$, where
	\[
	I^1_{\gamma,\mu} f(z) = \int_{d(z - w) < \alpha} \frac{f(w)}{d(z-w)^\mu}\ dw,
	\] and
	\[ I^2_{\gamma,\mu} f(z) = \int_{d(z - w) \geq \alpha} \frac{f(w)}{d(z-w)^\mu}\ dw. \]
	At this stage, let $\alpha > 0$ be a fixed constant. Its precise value is not essential for the analysis that follows.
	
	The first integral $I^1_{\gamma, \mu}f(z)$
	is the convolution of the function $f \in L^p(\mathbb{R}^N)$ with the kernel $ \frac{1}{d(z)^{\mu}}$, which is in  $L^1(B_\alpha)$ because $\mu< N_\gamma$. Therefore, this integral converges absolutely for almost every $z \in \mathbb{R}^N$. Similarly, the second integral $I^2_{\gamma,\mu} f(z)$
	is the convolution of $f \in L^p(\mathbb{R}^N)$  with the kernel $\frac{1}{d(z)^{\mu}}$. Note that, this kernel belongs to $ L^{p'}(\R^N \setminus B_\alpha)$, where $\frac{1}{p} + \frac{1}{p'} = 1$. Indeed, the integral
	\[
	\int_{d(z) \geq \alpha} d(z)^{-\mu p'} \ dz < \infty,
	\]
	provided $\mu p' > N_\gamma$, which is equivalent to $q < \infty$, where $ \frac{1}{q} = \frac{1}{p} + \frac{\mu}{N_\gamma} - 1$.
	
	Thus, both integrals are absolutely convergent a.e., and this establishes part $(a)$ of the theorem.
	
	Next, we aim to show that the operator $f \mapsto I_{\gamma, \mu} f$ is of weak type $(p, q)$, that is,
	\begin{equation}
		\label{1.3}
		\left| \{ z \in \R^N: |I_{\gamma, \mu} f(z)| > \kappa \}\right| \leq  \left( \frac{C \|f\|_p}{\kappa}\right)^q, \  f \in L^p(\R^N) \text{ and for all } \ \kappa >0.
	\end{equation}
	Thus, for any  $\kappa > 0$,
	\begin{align*}
		\left| \left\{ z \in \R^N : |I_{\gamma,\mu} f(z)| > 2\kappa \right\} \right| &\leq \left| \left\{ z\in \R^N : |I^1_{\gamma, \mu} f(z)| > \kappa \right\} \right|\\
		& \quad + \left| \left\{ z \in \R^N : |I^2_{\gamma,\mu} f(z)| > \kappa \right\} \right|.
	\end{align*}
	We note that it is sufficient to prove the inequality \eqref{1.3} with $2\kappa$ in place of $\kappa$ and $\| f \|_{p} = 1$.
	
	By Young's inequality \cite{Loss}, we have
	\[
	\| I^1_{\gamma, \mu} f \|_{p} \leq C \left( \int_{d(w) < \alpha} \frac{1}{d(w)^\mu} \ dw \right) \| f \|_{p} = C_1 \int_0^\alpha \rho^{N_\gamma-\mu-1} \ d\rho = C_1 \alpha^{N_\gamma-\mu}.
	\]
	Hence,
	\[
	\left| \left\{ z\in \R^N : |I^1_{\gamma,\mu} f(z)| > \kappa \right\} \right| \leq \frac{\| I^1_{\gamma,\mu} f \|_{p}^p}{\kappa^p} \leq C_1 \alpha^{p(N_\gamma-\mu)} \kappa^{-p}.
	\]
	On the other hand, by Young's inequality, we obtain
	\[
	\| I^2_{\gamma,\mu} f \|_{\infty} \leq C \left( \int_{d(w) \geq \alpha} \frac{1}{d(w)^{\mu p'}} \ dw \right)^{1/p'} \| f \|_{p} = C_2 \alpha^{-N_\gamma/q}.
	\]
	Choose $\alpha$ such that $C_2 \alpha^{-N_\gamma/q} = \kappa$. Then
	\[
	\left| \left\{ z \in \R^N : |I^2_{\gamma,\mu} f(z)| > \kappa \right\} \right| = 0,
	\]
	and
	\[
	\left| \left\{ z\in \R^N : |I^1_{\gamma,\mu} f(z)| > \kappa \right\} \right| \leq C_1 \alpha^{p(N_\gamma-\mu) }\kappa^{-p} = C_3 \kappa^{-q}.
	\]
	This establishes inequality \eqref{1.3}, which shows that the mapping $f \mapsto I_{\gamma,\mu} f $ is of weak type $(p, q)$.
	
	Now, in the special case $p = 1$, we have $q = \frac{N_\gamma}{\mu}$, and the above argument implies that $I_{\gamma, \mu}$ is of weak type $(1, q)$ with $q= \frac{N_\gamma}{\mu}$. This establishes part $(c)$ of the theorem.
	
	To conclude part (b), i.e. the strong-type boundedness condition
	\[
	\|I_{\gamma, \mu} f\|_{q} \leq C \|f\|_{p},
	\]
	we apply the Marcinkiewicz interpolation theorem.
	Let $1 < p < \frac{N_\gamma}{N_\gamma - \mu}$ and define $q$ by
	$\frac{1}{q} = \frac{1}{p} + \frac{\mu}{N_\gamma} - 1.$
	Choose $p_1, p_2$ such that
	$1 < p_1 < p < p_2 < \frac{N_\gamma}{N_\gamma - \mu},$
	and write
	\[
	\frac{1}{p} = \frac{\theta}{p_1} + \frac{1-\theta}{p_2}
	\quad \text{for some } \theta \in (0,1).
	\]
	Let $q_1, q_2$ be defined by
	\[
	\frac{1}{q_i} = \frac{1}{p_i} + \frac{\mu}{N_\gamma} - 1, \quad i=1,2.
	\]
	Since $I_{\gamma,\mu}$ is of weak type $(p_i,q_i)$ for $i=1,2$, the Marcinkiewicz interpolation theorem yields that it is of strong type $(p,q_0)$, where
	\[
	\frac{1}{q_0} = \frac{\theta}{q_1} + \frac{1-\theta}{q_2}.
	\]
	A direct computation shows that $q_0 = q$, which completes the proof.
	\end{proof}
	
	The following form of the Hardy-Littlewood-Sobolev inequality in the Grushin setting is a direct consequence of the previous theorem combined with an application of H\"older's inequality.
	
	\begin{prop}[Grushin HLS inequality]
		\label{Hardy}
		Let $p, r > 1$ and $0 < \mu < N_\gamma$ with $\frac{1}{p} + \frac{\mu}{N_\gamma} + \frac{1}{r} = 2$. Let $f \in L^{p}(\R^N)$ and $h \in L^{r}(\R^N)$. There exists a positive constant $C(p,N_\gamma, \mu, r)$, independent of $f, h$ such that
		\begin{equation*}
			\int_{\R^N} \int_{\R^N} \frac{f(z)h(w)}{d(z-w)^{\mu}}dz dw\leq C(p,N_\gamma, \mu, r)\|f\|_{p}\|h\|_{r}.
		\end{equation*}
	\end{prop}
	\begin{proof}
	Applying Fubini's theorem, we get
	\[
	\int_{\mathbb{R}^N} \int_{\mathbb{R}^N} \frac{f(z) h(w)}{d(z-w)^\mu} \ dz dw
	= \int_{\mathbb{R}^N} h(w) I_{\gamma,\mu} f(w) \ dw.
	\]
	By Theorem~\ref{lem1.2}\textnormal{(b)}, if \( \frac{1}{q} = \frac{1}{p} + \frac{\mu}{N_\gamma} - 1 \), then
	\[
	\|I_{\gamma, \mu} f\|_{q} \leq C \|f\|_{p}.
	\]
	Now apply H\"older's inequality with exponents $q$ and $r$, where $\frac{1}{r} + \frac{1}{q} = 1$. Then, we have
	\[
	\int_{\mathbb{R}^N} h(w) I_{\gamma,\mu} f(w) \ dw
	\leq \|h\|_{r} \|I_{\gamma,\mu} f\|_{q}
	\leq C \|f\|_{p} \|h\|_{r}.
	\]
	The condition $\frac{1}{r} + \frac{1}{q} = 1$ and $\frac{1}{q} = \frac{1}{p} +\frac{\mu}{N_\gamma} -1$ together give
	\[
	\frac{1}{p} + \frac{\mu}{N_\gamma} + \frac{1}{r} = 2.
	\]
	Hence, the HLS inequality holds with constant  $C = C(p, N_\gamma, \mu, r)$, completing the proof.
	\end{proof}
	
	We observe that, by Proposition \ref{Hardy}, the integral
	\[ \int_{\R^N} \left( \int_{\R^N} \frac{|u(w)|^{q}}{d(z-w)^{\mu}}dw \right) |u(z)|^{q} dz \]
	is well-defined provided that $|u|^q \in L^t(\R^N)$ for some $t>1$ satisfying $\frac{2}{t} +\frac{\mu}{N_\gamma} =2$.
	Now, if $u \in \mathcal{H}_\gamma(\mathbb{R}^N)$, then by the embedding
	$\mathcal{H}_\gamma(\mathbb{R}^N) \hookrightarrow L^{p}(\mathbb{R}^N),
	\quad p \in [2,2_\gamma^*],$
	we must have
	\[
	2 \le tq \le \frac{2N_\gamma}{N_\gamma - 2},
	\]
	equivalently,
	\[
	\frac{2N_\gamma - \mu}{N_\gamma}
	\le q \le
	\frac{2N_\gamma - \mu}{N_\gamma - 2}.
	\]  
	Therefore, $\frac{2N_\gamma - \mu}{N_\gamma}$ is called the lower critical exponent, and
	$2^*_{\gamma,\mu} := \frac{2N_\gamma - \mu}{N_\gamma - 2}$
	is called the upper critical exponent in the setting of the Grushin HLS inequality.
	To the best of our knowledge, neither the lower nor the upper critical case has been previously studied for Choquard-type nonlocal problems in the Grushin setting. In this work, we consider for the first time the upper critical exponent case. 
	
	From Proposition \ref{Hardy}, for all $u \in \mathcal{H}_\gamma(\mathbb{R}^N)$, we have
	\[
	\left(
	\int_{\mathbb{R}^N} \left(
	\int_{\mathbb{R}^N}
	\frac{|u(w)|^{2^*_{\gamma,\mu}}}{d(z-w)^{\mu}}\ dw \right)
	\,|u(z)|^{2^*_{\gamma,\mu}}\,dz
	\right)^{\frac{1}{2^*_{\gamma,\mu}}}
	\le
	C(N_\gamma,\mu)^{\frac{1}{2^*_{\gamma,\mu}}}
	\|u\|^2_{2^*_{\gamma}},
	\]
	where $C(N_\gamma,\mu)>0$ is a constant.
	
	From now on, we restrict ourselves to a bounded open subset $\Omega \subset \R^N$. From Proposition \ref{Hardy}, we conclude that  $J_\gamma$ is well-defined on $\mathcal{H}_\gamma(\Omega)$. It is easy to show that $J_\gamma\in C^1(\mathcal{H}_\gamma(\Omega))$ by standard techniques. The critical points of $J_\gamma$ coincide with the weak solutions of equation \eqref{a}.
	
	Next, we introduce the best constant associated with the Hardy-Littlewood-Sobolev inequality in the Grushin setting on $\Omega$, defined by
	
	\begin{equation}
		\label{1.5}
		S_\Omega:= \inf_{u\in \mathcal{H}_\gamma(\Omega) \setminus \{0\}} \frac{\displaystyle\int_{\Omega} |\nabla_\gamma u|^2 dz}{\Bigg( \displaystyle\int_{\Omega} \left( \displaystyle\int_{\Omega} \frac{|u(w)|^{\2}}{d(z-w)^{\mu}}dw \right) |u(z)|^{\2} dz\Bigg)^{\frac{1}{\2}}}.
	\end{equation}
	This constant plays a crucial role in establishing the Palais-Smale compactness condition for the functional $J_\gamma$.
	\begin{rem}
		To prove Theorem \ref{m}, we have assumed that
		\[   \lambda^* - \lambda < \frac{S_\Omega}{|\Omega|^{\frac{\2-2}{\2}}(\diam \Omega)^{\frac{\mu}{\2}}}.      \]
		Note that the term
		\[
		\frac{S_\Omega}{|\Omega|^{\frac{\2-2}{\2}}(\diam \Omega)^{\frac{\mu}{\2}}}
		\]
		depends on the domain $\Omega$. It is not guaranteed that $$\lambda^* < \frac{S_\Omega}{|\Omega|^{\frac{\2-2}{\2}}(\diam \Omega)^{\frac{\mu}{\2}}}.$$ If we take $\lambda<0$, then the above assumption is violated.
		Hence, $\lambda>0$ condition is considered throughout this article.
	\end{rem}
	\section{Palais-Smale sequence analysis}
	\setcounter{section}{3} \setcounter{equation}{0}
	\label{sec2}
	In this section, we prove the compactness condition for the Palais-Smale sequence. We begin by introducing a few preliminary lemmas essential for proving compactness.
	
	\begin{lem}
		\label{lem2.1}
		Let $N\geq 3$, $ 0 < \mu < N_\gamma$, and $\{u_n\}$ be a bounded sequence in $L^{2^*_{\gamma}}(\Omega)$ such that  $u_n \to u_0$ a. e. in $\Omega$ as $n \to \infty$. Then the following holds:
		\begin{align*}
			&\int_{\Omega} \left(\int_{\Omega} \frac{ |u_n(w)|^{\2}}{d(z-w)^\mu} dw \right) |u_n(z)|^{\2} dz\\
			& \quad- \int_{\Omega} \left(\int_{\Omega} \frac{ |u_n(w) - u_0(w)|^{\2}}{d(z-w)^\mu} dw\right) |u_n(z) - u_0(z)|^{\2} dz\\
			&\to \int_{\Omega} \left(\int_{\Omega} \frac{ |u_0(w)|^{\2}}{d(z-w)^\mu} dw \right)|u_0(z)|^{\2} dz\quad \text{as} \quad n \to \infty.
		\end{align*}
	\end{lem}
	\begin{proof}
	Since $\{u_n\}$ is bounded in $L^{2^*_\gamma}(\Omega)$ and $u_n \to u_0$ a. e. in $\Omega$, the Br\'ezis-Lieb Lemma \cite{Brezis.Lieb} implies that
	\begin{equation}
		\label{3.1}
		|u_n|^{2^*_{\gamma, \mu}} - |u_n - u_0|^{2^*_{\gamma,\mu}} \to |u_0|^{2^*_{\gamma, \mu}} \quad \text{ in } L^{\frac{2N_\gamma}{2N_\gamma - \mu}}(\Omega).
	\end{equation}
	We now apply Theorem \ref{lem1.2}(b) with
	\[ p = \frac{2N_\gamma}{2N_\gamma-\mu},  \quad  q = \frac{2N_\gamma}{\mu} \quad \text{ and } \quad f = |u_n|^{2^*_{\gamma, \mu}} - |u_n - u_0|^{2^*_{\gamma,\mu}}- |u_0|^{\2},\]
	and using equation \eqref{3.1}, we obtain
	\begin{align}
		\nonumber
		&\displaystyle\left\| \int_\Omega \frac{|u_n|^{2^*_{\gamma, \mu}} - |u_n - u_0|^{2^*_{\gamma,\mu}}- |u_0|^{\2}}{d(z-w)^\mu} \ dw \right\|_{\frac{2N_\gamma}{\mu}} \\
		& \leq C \left\| |u_n|^{2^*_{\gamma, \mu}} - |u_n - u_0|^{2^*_{\gamma,\mu}}- |u_0|^{\2} \right\|_{\frac{2N_\gamma}{2N_\gamma-\mu}} \to 0, \nonumber
	\end{align}
	i.e.
	\begin{equation}
		\label{3.3}
		\int_\Omega \frac{|u_n|^{2^*_{\gamma, \mu}} - |u_n - u_0|^{2^*_{\gamma,\mu}}}{d(z-w)^\mu} \ dw  \to  \int_\Omega \frac{|u_0|^{\2}}{d(z-w)^\mu} \ dw \quad \text{ in } \quad L^{\frac{2N_\gamma}{\mu}}(\Omega).
	\end{equation}
	
	Now, by Fubini's theorem, we have
	\begin{align}
		\label{3.4}
		&\int_{\Omega} \left( \int_\Omega\frac{ |u_n(w)|^{\2}}{d(z-w)^\mu} dw\right) |u_n|^{\2} \ dz \nonumber\\
		&- \int_{\Omega} \left(\int_\Omega\frac{|u_n(w) - u_0(w)|^{\2}}{d(z-w)^\mu} dw\right) |u_n(w) - u_0(w)|^{\2} \ dz \nonumber \\
		&= \int_{\Omega} \left(\int_\Omega \frac{  |u_n(w)|^{\2} - |u_n(w) - u_0(w)|^{\2}}{d(z-w)^\mu} dw\right)\nonumber\\
		& \quad\times \left(|u_n(z)|^{\2} - |u_n(z) - u_0(z)|^{\2} \right) dz \nonumber \\
		&\quad \quad+ 2 \int_{\Omega} \left(\int_\Omega \frac{|u_n(w)|^{\2} - |u_n(w) - u_0(w)|^{\2}}{d(z-w)^\mu} dw\right) |u_n(z) - u_0(z)|^{\2} \ dz.
	\end{align}
	From \eqref{3.1} and \eqref{3.3}, we know that the first term in \eqref{3.4} converges to
	\[
	\int_{\Omega} \left(\int_\Omega \frac{ |u_0(z)|^{\2}}{d(z-w)^\mu} dw\right) |u_0(z)|^{\2} \ dz.
	\]
	For the second term, we note that $|u_n - u_0|^{\2} \rightharpoonup 0$ in $L^{\frac{2N_\gamma}{2N_\gamma - \mu}}(\Omega)$, and combining this with \eqref{3.3}, we conclude that, the second term tends to 0 as $n \to \infty$. This completes the proof.
	\end{proof}
	\begin{lem}
		\label{lem2.3}
		Let $ N \geq 3$, $ 0 < \mu < N_\gamma $, and $\lambda > 0$. If $\{u_n\}\subset \mathcal{H}_\gamma(\Omega)$ is a Palais-Smale sequence of $J_\gamma$, then $\{u_n\}$ is bounded in  $\mathcal{H}_\gamma(\Omega)$. Moreover, if $u_0 \in \mathcal{H}_\gamma(\Omega)$ is the weak limit of $\{u_n\}$, then $u_0$ is a weak solution of the problem \eqref{a}.
	\end{lem}
	\begin{proof}
	By the definition of the Palais-Smale sequence, it is easy to observe that there exists $C_1 > 0$ such that
	\begin{align}
		\label{2.0}
		|J_\gamma(u_n)| \leq C_1 \quad \text{ and } \quad
		\left|\left\langle J'_\gamma(u_n), u_n \right\rangle_\gamma\right| \leq C_1 \varepsilon_n \|u_n\|_\gamma,
	\end{align}
	where $\varepsilon_n \to 0$ as $n \to \infty$.
	From \eqref{2.0}, we can obtain a constant $C_2>0$ such that
	\begin{align}
		\nonumber
		&J_\gamma(u_n)- \frac{1}{2} \left\langle J_\gamma'(u_n), u_n \right\rangle_\gamma \nonumber\\
		& = \left( \frac{1}{2}- \frac{1}{2 \cdot \2}\right)\int_{\Omega} \left( \int_{\Omega} \frac{ |u_n(w)|^{\2} }{d(z-w)^\mu} dw\right)|u_n(z)|^{\2} dz \leq C_2(1 + \varepsilon_n\|u_n\|_\gamma), \label{3.8}
	\end{align}
	that is
	\begin{equation}
		\label{2.1}
		\int_{\Omega} \left( \int_{\Omega} \frac{ |u_n(w)|^{\2} }{d(z-w)^\mu} dw\right)|u_n(z)|^{\2} dz \leq C_3(1 + \varepsilon_n\|u_n\|_\gamma),
	\end{equation}
	for some $C_3>0$. Let $\diam \Omega$ denotes the diameter of $\Omega$ with respect to the Grushin norm
	$d(\cdot)$, defined by
	\[   \diam \Omega = \sup\{ d(z-w): z,w \in \Omega\}. \]
	From \eqref{3.8}, we have
	\begin{align}
		C_2(1 + \varepsilon_n\|u_n\|_\gamma) & \geq \left( \frac{1}{2}- \frac{1}{2 \cdot \2}\right)\int_{\Omega} \left( \int_{\Omega} \frac{ |u_n(w)|^{\2} }{d(z-w)^\mu} dw\right)|u_n(z)|^{\2} dz\nonumber\\
		& \geq \left( \frac{1}{2}- \frac{1}{2 \cdot \2}\right) \frac{1}{(\diam\Omega)^\mu} \left( \int_\Omega |u_n(z)|^{\2} \ dz\right)^2. \nonumber
	\end{align}
	Thus, there exists a positive constant $C_4$ such that
	\begin{equation}
		\label{2.2}
		\|u_n\|_{\2}^{2 \cdot \2} \leq C_4(1 + \varepsilon_n\|u_n\|_\gamma).
	\end{equation}
	Since $\2 >2$, using H\"older's inequality and \eqref{2.2}, we have
	\begin{align}
		\|u_n\|_2^2 \leq  |\Omega|^{\frac{\2-2}{\2}} \|u_n\|_{\2}^2 \leq C_4^{\frac{1}{\2}} |\Omega|^{\frac{\2-2}{\2}} (1 + \varepsilon_n\|u_n\|_\gamma)^{\frac{1}{\2}}.
		\nonumber
	\end{align}
	It follows that we can find a positive constant $C_5$ such that
	\begin{equation}
		\label{2.4}
		\|u_n\|_2^2 \leq C_5 (1 + \varepsilon_n\|u_n\|_\gamma)^{\frac{1}{\2}}.
	\end{equation}
	From \eqref{2.0}, \eqref{2.1} and \eqref{2.4}
	\begin{align}
		C_1 & \geq J_\gamma(u_n)\nonumber\\
		& = \frac{1}{2} \|u_n\|_\gamma^2 -\frac{\lambda}{2} \|u_n\|_2^2 -\frac{1}{2\cdot \2} \int_\Omega \left( \int_\Omega \frac{|u_n(w)|^{\2}}{d(z-w)^\mu} \ dw \right) |u_n(z)|^{\2} \ dz \nonumber\\
		& \geq \frac{1}{2} \|u_n\|_\gamma^2 -\frac{\lambda C_5}{2} (1 + \varepsilon_n\|u_n\|_\gamma)^{\frac{1}{\2}} - \frac{C_3}{2\cdot \2} (1+ \varepsilon_n \|u_n\|_\gamma)\nonumber.
	\end{align}
	Hence, the sequence $\{u_n\}$ is bounded in $ \mathcal{H}_\gamma(\Omega)$ which implies that there exists a function $u_0 \in \mathcal{H}_\gamma(\Omega)$ such that, up to a subsequence $u_n \rightharpoonup u_0$ in $\mathcal{H}_\gamma(\Omega)$ and
	$u_n \rightharpoonup u_0$ in $L^{2^*_\gamma}(\Omega)$. Then, we have
	\begin{align*} &|u_n|^{\2} \rightharpoonup |u_0|^{\2} \text{ in }  L^{\frac{2N_\gamma}{2N_\gamma-\mu}}(\Omega) \\
		&\text{ and } |u_n|^{\2-2}u_n \rightharpoonup |u_0|^{\2-2}u_0 \text{ in }  L^{\frac{2N_\gamma}{N_\gamma-\mu+2}}(\Omega) \end{align*}
	as $ n \to \infty $. By the Hardy-Littlewood-Sobolev inequality, $\int_\Omega \frac{|u_n(w)|^{\2}}{d(z-w)^\mu} \ dw$ defines a linear continuous map from $L^{\frac{2N_\gamma}{2N_\gamma-\mu}}(\Omega)$ to $L^{\frac{2N_\gamma}{\mu}}(\Omega)$. This implies that
	\[ \int_\Omega \frac{ |u_n(w)|^{\2}}{d(z-w)^\mu}\ dw  \rightharpoonup \int_\Omega \frac{ |u_0(w)|^{\2}}{d(z-w)^\mu}\ dw  \quad \text{in} \quad L^{\frac{2N_\gamma}{\mu}}(\Omega) \quad \text{as} \quad n \to \infty.
	\]
	Combining all these, we obtain
	\begin{align}
		& \left(\int_\Omega \frac{ |u_n(w)|^{\2}}{d(z-w)^\mu}\ dw \right) |u_n(z)|^{\2-2}u_n(z) \nonumber\\
		& \rightharpoonup \left(\int_\Omega \frac{ |u_0(w)|^{\2}}{d(z-w)^\mu}\ dw\right) |u_0(z)|^{\2-2}u_0(z) \quad \text{in} \quad L^{\frac{2N_\gamma}{N_\gamma+2}}(\Omega) \quad  \nonumber
	\end{align}
	as $n \to \infty$. Next, we know that $J_\gamma'(u_n) \to 0$ as $n \to \infty$. Thus, for any \( \varphi \in \mathcal{H}_\gamma(\Omega) \), we obtain
	\begin{align*}
		&\int_{\Omega} \nabla_\gamma u_0 \nabla_\gamma \varphi \ dz - \lambda \int_{\Omega} u_0 \varphi \ dz \\
		&- \int_{\Omega} \left(\int_{\Omega} \frac{|u_0(w)|^{\2} }{d(z-w)^\mu} \ dw \right) |u_0(z)|^{\2-2}u_0(z) \varphi(z) \ dz=0,
	\end{align*}
	which implies that $u_0$ is a weak solution of equation \eqref{a}.
	This completes the proof.
	\end{proof}
	
	Let $u_0$ be the weak solution as proved in the previous lemma, taking $\varphi = u_0 \in \mathcal{H}_\gamma(\Omega)$ as a test function in \eqref{a}, we get
	\[
	\int_{\Omega} |\nabla_\gamma u_0|^2 \ dz = \lambda \int_{\Omega} u_0^2 \ dz - \int_{\Omega} \left(\int_{\Omega} \frac{ |u_0(w)|^{\2}}{d(z-w)^\mu} \ dw \right) |u_0(z)|^{\2} \ dz.
	\] This implies that
	\[
	J_\gamma(u_0) = \frac{N_\gamma+ 2 - \mu }{4N_\gamma - 2\mu} \int_{\Omega} \int_{\Omega} \frac{|u_0(z)|^{\2} |u_0(w)|^{\2}}{d(z-w)^\mu} \, dz \, dw \geq 0.
	\]
	\begin{lem}
		\label{lem2.4}
		Let $ N \geq 3$, $ 0 < \mu < N_\gamma $, and $\lambda > 0$. If $\{u_n\}$ is a Palais-Smale sequence of $J_\gamma$ at level $c$ with
		\[
		c < \frac{N_\gamma + 2 - \mu}{4N_\gamma - 2\mu } S^{\frac{2N_\gamma - \mu}{N_\gamma + 2 - \mu}}_{\Omega},
		\]
		then the functional $J_\gamma$ satisfies the Palais-Smale condition at level $c$.
	\end{lem}
	
	\begin{proof}
	From Lemma \ref{lem2.3}, there exists $u_0$ in $\mathcal{H}_\gamma(\Omega)$ such that $u_n \rightharpoonup u_0$ in $\mathcal{H}_\gamma(\Omega)$, and let us define $v_n := u_n - u_0$. It follows that $v_n \rightharpoonup 0$ in $\mathcal{H}_\gamma(\Omega)$ and $v_n \to 0$ almost everywhere in $\Omega$. Additionally, by the Brezis-Lieb lemma \cite{Brezis.Lieb}, it is easy to see that
	\[
	\int_{\Omega} |\nabla_\gamma u_n|^2 \, dz = \int_{\Omega} |\nabla_\gamma v_n|^2 \, dz + \int_{\Omega} |\nabla_\gamma u_0|^2 \, dz + o_n(1).
	\]
	From Lemma \ref{lem2.1}, we have
	\begin{align}
		&\int_{\Omega} \left(\int_{\Omega} \frac{ |u_n(w)|^{\2}}{d(z-w)^\mu} \ dw \right) |u_n(z)|^{\2} \ dz \nonumber\\
		&= \int_{\Omega} \left(\int_{\Omega} \frac{ |v_n(w)|^{\2}}{d(z-w)^\mu} \ dw \right) |v_n(z)|^{\2} \ dz \nonumber\\
		& \quad+ \int_{\Omega} \left(\int_{\Omega} \frac{ |u_0(w)|^{\2}}{d(z-w)^\mu} \ dw \right)|u_0(z)|^{\2} \ dz + o_n(1).\nonumber
	\end{align}
	Since $J_\gamma(u_n) \to c$ as $n \to \infty$ and by the above equations, we have
	\begin{align}
		c + o_n(1)&=  J_\gamma(u_n) + o_n(1)\nonumber\\
		&= \frac{1}{2} \int_{\Omega} |\nabla_\gamma u_n|^2 \ dz - \frac{\lambda}{2} \int_{\Omega} u_n^2 \ dz \nonumber\\
		& \quad- \frac{1}{2 \cdot \2} \int_{\Omega} \left( \int_{\Omega} \frac{ |u_n(w)|^{\2} }{d(z-w)^\mu} \ dw \right) |u_n(z)|^{\2} \ dz + o_n(1)\nonumber\\
		&= \frac{1}{2} \int_{\Omega} |\nabla_\gamma v_n|^2 \ dz + \frac{1}{2} \int_{\Omega} |\nabla_\gamma u_0|^2 \ dz \nonumber\\
		&\quad - \frac{1}{2 \cdot \2} \int_{\Omega} \left(\int_{\Omega} \frac{ |v_n(w)|^{\2}}{d(z-w)^\mu}  \ dw \right) |v_n(z)|^{\2} \ dz
		\nonumber\\
		&\quad \quad- \frac{\lambda}{2} \int_{\Omega} v_n^2  dz -\frac{\lambda}{2} \int_{\Omega} u_0^2  dz\nonumber\\
		& \quad \quad- \frac{1}{2 \cdot \2} \int_{\Omega} \left( \int_{\Omega} \frac{ |u_0(w)|^{\2}}{d(z-w)^\mu} dw \right) |u_0(z)|^{\2} dz + o_n(1)\nonumber.
	\end{align}
	Therefore,
	\begin{align}
		c + o_n(1)
		&= J_\gamma(u_0) + \frac{1}{2} \int_{\Omega} |\nabla_\gamma v_n|^2  dz - \frac{\lambda}{2} \int_{\Omega} v_n^2 dz \nonumber\\
		&- \frac{1}{2 \cdot \2} \int_{\Omega}\left( \int_{\Omega} \frac{ |v_n(w)|^{\2}}{d(z-w)^\mu} dw\right)|v_n(z)|^{\2} dz + o_n(1).\nonumber
	\end{align}
	By the compact embedding $\mathcal{H}_\gamma(\Omega) \hookrightarrow L^2(\Omega)$, we have $\int_{\Omega} v_n^2 \ dz \to 0$ as $n \to \infty$ and since $ J_\gamma(u_0) \geq 0$, therefore
	\begin{align}
		c + o_n(1)
		&\geq \frac{1}{2} \int_{\Omega} |\nabla_\gamma v_n|^2 \, dz - \frac{1}{2 \cdot \2} \int_{\Omega} \left( \int_{\Omega} \frac{ |v_n(w)|^{\2}}{d(z-w)^\mu} \ dw \right) |v_n(z)|^{\2} \ dz \nonumber\\
		& \quad + o_n(1). \label{2.3}
	\end{align}
	As $u_0$ is a weak solution this implies that $\langle J'_\gamma(u_0), u_0 \rangle_\gamma = 0$ with $\int_{\Omega} v_n^2 \ dz \to 0$ as $n \to \infty$, we have
	\begin{align}
		o_n(1) &= \langle J'_\gamma(u_n), u_n \rangle_\gamma\nonumber\\
		& = \int_{\Omega} |\nabla_\gamma u_n|^2 \ dz - \lambda \int_{\Omega} u_n^2 \ dz - \int_{\Omega} \left(\int_{\Omega} \frac{ |u_n(w)|^{\2}}{d(z-w)^\mu} \ dw \right) |u_n(z)|^{\2} \ dz \nonumber\\
		&= \int_{\Omega} |\nabla_\gamma v_n|^2 \ dz + \int_{\Omega} |\nabla_\gamma u_0|^2 \ dz - \lambda \int_{\Omega} u_0^2 \ dz\nonumber\\
		& \quad - \int_{\Omega} \left(\int_{\Omega} \frac{ |v_n(w)|^{\2} }{d(z-w)^\mu} \ dw\right) |v_n(z)|^{\2} \ dz\nonumber\\
		&\quad- \int_{\Omega} \left(\int_{\Omega} \frac{ |u_0(w)|^{\2}}{d(z-w)^\mu} \ dw \right) |u_0(z)|^{\2} \ dz + o_n(1)\nonumber\\
		&= \langle J'_\gamma(u_0), u_0 \rangle_\gamma + \int_{\Omega} |\nabla_\gamma v_n|^2 \ dz\nonumber\\
		&\quad - \int_{\Omega}\left( \int_{\Omega} \frac{ |v_n(w)|^{\2}}{d(z-w)^\mu} \ dw \right)|v_n(z)|^{\2} \ dz + o_n(1)\nonumber\\
		&= \int_{\Omega} |\nabla_\gamma v_n|^2 \, dz - \int_{\Omega}\left( \int_{\Omega} \frac{ |v_n(w)|^{\2}}{d(z-w)^\mu} \ dw \right) |v_n(z)|^{\2} \ dz + o_n(1). \nonumber
	\end{align}
	This implies that there must exist a non-negative constant $b$ such that
	\[
	\int_{\Omega} |\nabla_\gamma v_n|^2 \, dz \to b \quad \text{and} \quad \int_{\Omega}\left( \int_{\Omega} \frac{ |v_n(w)|^{\2}}{d(z-w)^\mu} \ dw \right) |v_n(z)|^{\2} \ dz \to b,
	\]
	as  $n \to \infty$. Thus, from \eqref{2.3}, we get
	\begin{equation}
		c \geq \frac{N_\gamma + 2- \mu}{4N_\gamma - 2\mu} b. \label{2.5}
	\end{equation}
	Recall the definition of the best constant $S_{\Omega}$ in \eqref{1.5}, we have
	\[
	S_{\Omega} \left( \int_{\Omega} \left( \int_{\Omega} \frac{ |v_n(w)|^{\2}}{d(z-w)^\mu} \ dw \right)|v_n(z)|^{\2} \ dz \right)^{\frac{ N_\gamma - 2}{2N_\gamma - \mu}} \leq \int_{\Omega} |\nabla_\gamma v_n|^2 \, dz,
	\]
	which gives that
	\[
	b \geq S_{\Omega} b^{\frac{N_\gamma - 2}{2N_\gamma - \mu}}.
	\]
	From this we conclude that $$ \textrm{ either } b = 0\ \ \textrm{ or}\ \  b \geq S_{\Omega}^{ \frac{2N_\gamma - \mu }{N_\gamma -\mu + 2}}. $$ If $ b = 0 $, the proof is completed. If $$b \geq S_{\Omega}^{ \frac{2N_\gamma - \mu }{N_\gamma -\mu + 2}},  $$ then from \eqref{2.5}, we deduce that
	\[
	\frac{N_\gamma - \mu +2}{4 N_\gamma- 2\mu} S_{\Omega}^{ \frac{2N_\gamma - \mu}{N_\gamma -\mu + 2}}  \leq \frac{ N_\gamma - \mu+2}{4N_\gamma- 2\mu} b \leq c,
	\]
	which contradicts the hypothesis that
	\[
	c < \frac{N_\gamma - \mu +2}{4N_\gamma- 2\mu} S_{\Omega}^{ \frac{2N_\gamma - \mu}{N_\gamma -\mu + 2}}.
	\]
	Hence, $ b = 0 $ which further implies that $\| u_n - u_0 \|_\gamma \to 0  \text{ as }  n \to \infty$.
	This concludes the proof.
	\end{proof}
	
	\section{Proof of Theorem \ref{m}}
	\setcounter{section}{4} \setcounter{equation}{0}
	\label{sec3}
	We are now prepared to start the proof of Theorem \ref{m}. The following lemma is very crucial proving the main result. Before starting, we fix some notations. Since $\lambda^* = \min\{\lambda_k : \lambda < \lambda_k\}$ defined in Theorem \ref{m}, we assume that $\lambda^* = \lambda_k$ for some $k \in \mathbb{N}$. It is known that $\lambda^*$ has multiplicity $m \in \mathbb{N}$ by assumption. Let $l \geq 1$ denote the multiplicity of the first eigenvalue $\lambda_1$. Therefore, we can write
	\[     \lambda^* = \lambda_1 = \cdots = \lambda_m < \lambda_{m+1} , \ \text{ if } k=1    \]
	\[       \text{and } \  \lambda_{l+ k-1} < \lambda^* = \lambda_{l+k} = \cdots = \lambda_{l+k+m-1}< \lambda_{l+k+m}, \ \text{ if } \ k\geq 2. \]
	With the notation from Theorem \ref{m}, we define
	\begin{equation*}
		W:= \left\{
		\begin{aligned}
			& \text{span}\{\phi_1, \dots, \phi_{m}\}  \quad  &&\text{if} \ k=1,\\
			& \text{span}\{\phi_1, \dots, \phi_{l+k+m-1} \} \quad &&\text{if} \ k\geq 2
		\end{aligned}
		\right.
	\end{equation*}
	\begin{equation*}
		\text{and } \ V := \left\{
		\begin{aligned}
			&  \mathcal{H}_\gamma(\Omega) \quad  &&\text{if} \ k=1,\\
			& \left\{ u \in \mathcal{H}_\gamma(\Omega) : \langle u, \phi_j \rangle_\gamma = 0 \ \forall j = 1, \cdots, l +k-1 \right\} \quad &&\text{if} \ k\geq 2.
		\end{aligned}
		\right.
	\end{equation*}
	{%
		Note that both $W$ and $V$ are closed subsets of $ \mathcal{H}_\gamma(\Omega)$, with
		\[
		\dim W = l + k + m - 1 \quad \text{and} \quad \text{codim} \ V = l +k - 1, 
		\] }
	hence condition (iii)(c) of Theorem \ref{thm1} is satisfied. In the next lemma, we verify the remaining conditions of Theorem \ref{thm1}.
	
	\begin{lem}\label{lem3.1}
		Assume that $\lambda^*$ is defined as Theorem \ref{m} with multiplicity $m\in \mathbb{N}$. Then, there exist positive constants $\beta',\delta$ and $\rho$ such that
		\begin{enumerate}
			\item $J_\gamma(u) \leq \beta'$, for all $u\in W$;
			\item $J_\gamma(u) \geq \delta$, for any $u\in V$ with $\|u\|_\gamma =\rho$;
		\end{enumerate}
		where $V$ and $W$ are given as above.
	\end{lem}
	\begin{proof}
	(1) Let $u \in W$, and let $\lambda^*= \lambda_{l+k}, k \in \mathbb{N}$, we can write
	\begin{equation*}
		u(z)= \sum_{j=1}^{l+k+m-1} u_j \phi_j(z)
	\end{equation*}
	with $u_j\in \R, j= 1,2, \cdots, l +k+m-1$.
	
	Since $\{ \phi_1, \phi_2, \cdots\}$ is an orthonormal basis of $L^2(\Omega)$ and an orthogonal basis of $\mathcal{H}_\gamma(\Omega)$,  we have
	\begin{equation*}
		\|u\|^2_\gamma = \sum_{j=1}^{l+k+m-1} u_j^2 \|\phi_j\|^2_\gamma = \sum_{j=1}^{l+k+m-1} \lambda_j u_j^2 \leq \lambda_{l+k} \sum_{j=1}^{l+k+m-1} u_j^2= \lambda_{l+k} \|u\|^2_2 = \lambda^* \|u\|^2_2.
	\end{equation*}
	In a similar mannar one can verify the above relation for $\lambda^* = \lambda_1$.
	Let $\diam \Omega$ denotes the diameter of the domain $\Omega$ with respect to the Grushin norm $d(\cdot)$, then for $u\in W$
	\begin{align}
		J_\gamma(u) & = \frac{1}{2} \|u\|^2_\gamma-\frac{\lambda}{2} \|u\|^2_2 - \frac{1}{2\cdot\2} \int_\Omega \left( \int_\Omega \frac{|u(w)|^{\2}}{d(z-w)^\mu}dw\right)|u(z)|^{\2}dz \nonumber\\
		& \leq \frac{1}{2}(\lambda^* - \lambda) \|u\|_2^2 - \frac{1}{2\cdot\2} \int_\Omega \left( \int_\Omega \frac{|u(w)|^{\2}}{d(z-w)^\mu}dw\right)|u(z)|^{\2}dz \nonumber\\
		& \leq \frac{1}{2}(\lambda^* - \lambda) \|u\|_2^2 - \frac{1}{2(\diam \Omega)^\mu\2} \int_\Omega \left( \int_\Omega |u(w)|^{\2}dw\right)|u(z)|^{\2} dz \nonumber\\
		& \leq \frac{1}{2}(\lambda^* - \lambda) \|u\|_2^2 - \frac{1}{2 (\diam \Omega)^\mu \2} \left(\int_\Omega|u(z)|^{\2} dz \right)^2 \nonumber\\
		& \leq \frac{1}{2}(\lambda^* - \lambda) |\Omega|^{\frac{\2-2}{\2}} \|u\|_{\2}^2 - \frac{1}{2 (\diam \Omega)^\mu \2} \|u\|_{\2}^{2\cdot \2}. \label{3.2}
	\end{align}
	Let $$g(\tau) = \frac{1}{2}(\lambda^* - \lambda)|\Omega|^{\frac{\2-2}{\2}} \tau^2 - \frac{1}{2(\diam \Omega)^\mu\2} \tau^{2\cdot \2},\ \ \textrm{for}\ \ \tau \geq 0.$$
	Then $g(\tau)$ has a maximum at $$\tau_{\max} = \Big(  (\lambda^*-\lambda)|\Omega|^{\frac{\2-2}{\2}}(\diam \Omega)^\mu\Big)^{\frac{1}{2(\2-1)}}.$$
	Therefore, by simple computations, we have
	\begin{align*}
		g(\tau_{max})= \1 \left(1-\frac{1}{\2}\right)\left[  (\lambda^*-\lambda)^{\2} |\Omega|^{\2-2}(\diam \Omega)^{\mu}\right]^{\frac{1}{\2-1}}.
	\end{align*}
	Hence, we choose $$\beta' = \1 \left(1-\frac{1}{\2}\right)\left[  (\lambda^*-\lambda)^{\2} |\Omega|^{\2-2}(\diam \Omega)^{\mu}\right]^{\frac{1}{\2-1}}.$$ Then, from \eqref{3.2} we have
	\begin{equation*}
		J_\gamma(u) \leq \beta' \text{ for all } u \in W.
	\end{equation*}
	
	(2) For every $u \in V$, it follows from the variational characterization of $ \lambda^* = \lambda_{l+k}$ given by
	\[
	\lambda_{l+k} = \min_{u \in V \setminus \{ 0 \}} \frac{\| u \|^2_\gamma}{\| u \|^2_2},
	\] the following inequality holds
	\[
	\| u \|^2_\gamma \geq \lambda^* \| u \|^2_2.
	\] 
	By the characterization of $\lambda_1$, this inequality holds for $\lambda= \lambda_1$ case also.
	Therefore, by the definition of $S_{\Omega}$ and noting that $\lambda > 0$, we obtain
	\begin{align}
		J_\gamma(u) &\geq \frac{1}{2} \left( 1 - \frac{\lambda}{\lambda^*} \right) \| u \|^2_\gamma - \frac{1}{2\cdot \2} \int_\Omega \left( \int_\Omega \frac{|u(w)|^{\2}}{d(z-w)^\mu}dw\right)|u(z)|^{\2}dz\nonumber\\
		& \geq \frac{1}{2} \left( 1 - \frac{\lambda}{\lambda^*} \right) \| u \|^2_\gamma - \frac{1}{2\cdot \2 S_{\Omega}^{2\cdot \2}} \|u\|_\gamma^{2\cdot\2}\nonumber\\
		& = \|u\|_\gamma^2 \left( \frac{1}{2} \left( 1 - \frac{\lambda}{\lambda^*} \right) - \frac{1}{2\cdot \2 S_{\Omega}^{2\cdot \2}} \|u\|_\gamma^{2\cdot\2-2}\right) \nonumber
	\end{align}

	Now, let $u \in V$ such that $\| u \|_\gamma = \rho > 0$. Since $2\cdot \2 > 2$, we can choose $\rho$ sufficiently small, say $\rho \leq \overline{\rho}$ with $\overline{\rho}>0$ such that
	\[
	\frac{1}{2} \left( 1 - \frac{\lambda}{\lambda^*} \right) - \frac{1}{2\cdot \2 S_{\Omega}^{2\cdot \2}} \| u \|^{\2-2}_\gamma > 0,
	\]
	and
	\begin{align}
		\rho^2 \left( \frac{1}{2} \left( 1 - \frac{\lambda}{\lambda^*} \right) -  \frac{1}{2\cdot \2 S_{\Omega}^{2\cdot \2}}\rho^{\2-2} \right) < \rho^2 \left( \frac{1}{2} \left( 1 - \frac{\lambda}{\lambda^*} \right) \right) < \beta'.\nonumber
	\end{align}
	\end{proof}
	
	\subsection{Proof of Theorem \ref{m}} By Lemma \ref{lem2.4}, the functional $ J_{ \gamma}$ satisfies compactness assumption (iii) of Theorem \ref{thm1} with
	\[ \beta =  \frac{N_\gamma - \mu +2}{4 N_\gamma- 2\mu} S_{\Omega}^{ \frac{2N_\gamma - \mu}{N_\gamma -\mu + 2}}.\]
	
	From the Lemma \ref{lem3.1}, $J_\gamma$ verifies geometric properties required by the assumption (iii)-(a) and (b) of Theorem \ref{thm1} with
	\begin{align*} \rho &= \overline{\rho}, \\
		\beta' &= \1 \left(1-\frac{1}{\2}\right)\left[  (\lambda^*-\lambda)^{\2} |\Omega|^{\2-2}(\diam \Omega)^{\mu}\right]^{\frac{1}{\2-1}} \\
		\text{ and } \delta &= \rho^2 \left( \frac{1}{2} \left( 1 - \frac{\lambda}{\lambda^*} \right) -  \frac{1}{2\cdot \2 S_{\Omega}^{2\cdot \2}}\rho^{\2-2} \right).
	\end{align*}
	By our assumption $$\lambda^* -\lambda <  \frac{S_{\Omega}}{|\Omega|^{\frac{\2-2}{\2}}(\diam \Omega)^{\frac{\mu}{\2}}}, $$
	it is easy to show that
	\begin{equation*}
		\1 \left(1-\frac{1}{\2}\right)\left[  (\lambda^*-\lambda)^{\2} |\Omega|^{\2-2}(\diam \Omega)^{\mu}\right]^{\frac{1}{\2-1}} < \frac{N_\gamma - \mu +2}{4N_\gamma- 2\mu} S_{\Omega}^{ \frac{2N_\gamma - \mu}{N_\gamma -\mu + 2}}.
	\end{equation*}
	Therefore, assumption (iii) of Theorem \ref{thm1} holds with $\delta < \beta'<\beta$. We conclude that $J_\gamma$ has $m$ pairs of critical points $\{ -u_{\gamma,i}, u_{\gamma,i}\}$ with critical values in $[\delta, \beta']$. These critical points are nontrivial as $J_\gamma(0)=0$ and
	\begin{equation}
		\label{4.1}
		0< \delta \leq J_\gamma(\pm u_{\gamma,i})\leq \beta'.
	\end{equation}
	Fix an index $i \in \{1, \ldots, m\}$. Using equation \eqref{4.1}, we have
	\begin{align}
		\label{4.2}
		\beta' &\geq J_\gamma(u_{\gamma,i}) \nonumber\\
		&= J_{\gamma}(u_{\gamma,i}) - \frac{1}{2} \langle J'_\gamma(u_{\gamma,i}), u_{\gamma,i} \rangle_\gamma \nonumber \\
		&= \left( \frac{1}{2} - \frac{1}{2\cdot \2} \right) \int_\Omega \left( \int_\Omega \frac{|u_{\gamma,i}(z)|^{\2} }{d(z-w)^\mu} \ dw\right) |u_{\gamma,i}(z)|^{\2} \ dz \\
		&\geq \left(\frac{1}{2}- \frac{1}{2\cdot \2} \right) \frac{1}{(\diam \Omega)^\mu} \|u_{\gamma,i}\|_{\2}^{2\cdot \2} \nonumber
	\end{align}
	Taking the limit as $\lambda \to \lambda^*$, we observe that $\beta' \to 0$. Hence, from the above inequality, it follows that
	\begin{equation}
		\label{4.3}
		\|u_{\gamma,i}\|_{\2} \to 0 \quad \text{as } \lambda \to \lambda^*.
	\end{equation}
	Since the embedding $L^{\2}(\Omega) \hookrightarrow L^2(\Omega)$ is continuous due to the boundedness of $\Omega$ and $\2 >2$, then by \eqref{4.3}, we have
	\begin{equation}
		\label{4.4}
		\|u_{\gamma,i}\|_2 \to 0 \quad \text{as } \lambda \to \lambda^*.
	\end{equation}
	Now, again from \eqref{4.1}, we can write
	\begin{align*}
		\beta' \geq  J_{\gamma}(u_{\gamma,i}) &= \frac{1}{2} \|u_{\gamma,i}\|_\gamma - \frac{\lambda}{2} \|u_{\gamma,i}\|_2 \\
		& - \frac{1}{2 \cdot \2}\int_\Omega \left( \int_\Omega \frac{|u_{\gamma,i}(z)|^{\2} }{d(z-w)^\mu} \ dw\right) |u_{\gamma,i}(z)|^{\2} \ dz.
	\end{align*}
	
	Using \eqref{4.2} and \eqref{4.4}, we conclude that
	\[
	\|u_{\gamma,i}\|_{\gamma} \to 0 \quad \text{as } \lambda \to \lambda^*.
	\]
	This completes the proof of Theorem \ref{m}.

	\section*{acknowledgements}
		Suman Kanungo acknowledges the financial aid from CSIR, Govt. of India, File No. 09/1237(15789)/2022-EMR-I. Pawan Kumar Mishra is supported by the Science and Engineering Research Board, Government of India, Grant No. MTR/2022/000495. This work has been funded by the European Union - NextGenerationEU within the framework of PNRR  Mission 4 - Component 2 - Investment 1.1 under the Italian Ministry of University and Research (MUR) program PRIN 2022 - grant number 2022BCFHN2 - Advanced theoretical aspects in PDEs and their applications - CUP: H53D23001960006.

\end{document}